\documentclass[a4paper, 12pt]{article}

\usepackage[margin=1in]{geometry}
\usepackage{amsmath,amssymb,amsthm,mathtools}
\usepackage{microtype}
\usepackage{enumitem}
\usepackage[colorlinks=true,linkcolor=blue,citecolor=blue,urlcolor=blue]{hyperref}

\newtheorem{theorem}{Theorem}[section]

\newtheorem{lemma}[theorem]{Lemma}
\newtheorem{corollary}[theorem]{Corollary}

\theoremstyle{definition}
\newtheorem{definition}[theorem]{Definition}

\DeclareMathOperator{\ord}{ord}
\DeclareMathOperator{\supp}{supp}

\title{The Gao-Zhuang conjecture for the Heisenberg group over $\mathbb{F}_p$}
\author{
Yongke Qu
\qquad
Guoqing Wang\thanks{Corresponding author. E-mail: gqwang1979@aliyun.com.}}
\date{}

\begin{document}

\maketitle

\begin{abstract}
Let $G$ be a finite nonabelian group. The small Davenport constant
$\mathsf d(G)$ of $G$ is the largest integer $\ell$ such that there
exists a product-one free sequence over $G$ of length $\ell$, while
the Gao constant $E(G)$ of $G$ is the least integer $\ell$ such that
every sequence over $G$ of length at least $\ell$ contains a
product-one subsequence of length exactly $|G|$. A long-standing
conjecture of Gao and Zhuang \cite{ZG2005} asserts that
$E(G)=\mathsf d(G)+|G|$ for every finite nonabelian group $G$.

Let $p$ be an odd prime and let
$H_{p^3}=\operatorname{UT}_3(\mathbb F_p)$ be the finite Heisenberg
group over $\mathbb F_p$. Godara and Sarkar proved the Gao-Zhuang
equality for $H_{27}=\operatorname{UT}_3(\mathbb F_3)$ and asked
whether the same equality holds for $H_{p^3}$ for every odd prime
$p$. Recently, Volkmann proved that
$\mathsf d(H_{p^3})=3p-3$. In this paper, we determine the Gao
constant of $H_{p^3}$ and prove that
$E(H_{p^3})=\mathsf d(H_{p^3})+|H_{p^3}|=p^3+3p-3$. Together with the known abelian and cyclic-index cases, this
completes the verification of the Gao-Zhuang equality for all groups
of order $p^3$, for every prime $p$.
\end{abstract}

\noindent\textbf{Keywords.}
Zero-sum; Gao-Zhuang conjecture;  Gao constant; small Davenport constant;
Heisenberg group; finite $p$-group.

\medskip
\noindent\textbf{2020 Mathematics Subject Classification.}
Primary 11B75; Secondary 20D60.

\section{Introduction}

The classical theorem of Erd\H{o}s, Ginzburg and Ziv \cite{EGZ}
asserts that every sequence of $2n-1$ elements of the cyclic group
$C_n$ contains an $n$-term zero-sum subsequence. It is one of the
foundational results of zero-sum theory; see the survey of Gao and
Geroldinger \cite{GaoGeroldingersurvey}. For a finite group $G$, a
nonempty subsequence is called \emph{product-one} if its terms can be
ordered so that their product is the identity. The small Davenport
constant $\mathsf d(G)$ is the largest length of a product-one free
sequence over $G$, while the Gao constant $E(G)$ is the least integer
$\ell$ such that every sequence over $G$ of length at least $\ell$
contains a product-one subsequence of length $|G|$.

For finite abelian groups, Gao \cite{Gao1996} proved the fundamental
identity $E(G)=\mathsf d(G)+|G|$. Motivated by this result, Gao and
Zhuang \cite{ZG2005} conjectured that the same equality holds for every
finite group. In contrast with the abelian case, the nonabelian
problem involves both the selection of a subsequence and the ordering
of its terms, and the conjecture is known only for particular classes
of groups. Gao and Zhuang initiated the study for dihedral groups of
large prime index \cite{ZG2005}. Bass \cite{Bass2007} subsequently
verified the conjecture for all dihedral and dicyclic groups and for
all nonabelian groups of order $pq$, where $p$ and $q$ are primes.
Han \cite{Han2015} determined the Gao constant for
$C_p\ltimes C_{pn}$;  Han and Zhang \cite{HanZhang2019} obtained
further results for semidirect products of the form
$C_m\ltimes_{\varphi}C_{mn}$. Further progress for nonabelian
metacyclic groups was obtained by Qu and Li \cite{QuLi2023} and by
Avelar, Brochero Mart\'{\i}nez and Ribas
\cite{AvelarBrocheroRibas2023}. More recently, Oh, Ribas, Zhao and
Zhong \cite{ORZZ} completed the determination of the Gao constant,
together with the associated inverse problem, for all metacyclic
groups of the form $C_n\rtimes_s C_2$. Qu, Gao and Li \cite{QuGaoLi} verified the
Gao-Zhuang conjecture for every finite group having a cyclic
subgroup whose index is the smallest prime divisor of the group
order.

The finite Heisenberg group
$H_{p^3}=\operatorname{UT}_3(\mathbb F_p)$ is a canonical and
fundamental example among finite nonabelian $p$-groups. For every odd
prime $p$, it is the unique nonabelian group of order $p^3$ and
exponent $p$. Moreover, it is nilpotent of class $2$, with
$Z(H_{p^3})=[H_{p^3},H_{p^3}]\cong C_p$ and
$H_{p^3}/Z(H_{p^3})\cong C_p^2$; in particular, it is the extraspecial
$p$-group of order $p^3$ and exponent $p$ \cite{Gorenstein}. These
features make $H_{p^3}$ one of the most basic testing grounds for
questions on nonabelian $p$-groups, while its elementary abelian
central quotient provides a particularly natural bridge between
nonabelian product-one theory and classical zero-sum theory.

Godara and Sarkar \cite{GodaraSarkar}  verified the Gao-Zhuang conjecture
for $H_{27}$ and explicitly asked whether it holds for $H_{p^3}$ for
every odd prime $p$. Very recently, Volkmann \cite{Volkmann} proved
the uniform formula $\mathsf d(H_{p^3})=3p-3$. Thus the remaining
problem is to determine $E(H_{p^3})$. Our main result answers this
question affirmatively for the whole family.

\begin{theorem}\label{theorem:main}
Let $p$ be an odd prime and let $G=H_{p^3}$ be the Heisenberg group. Then
$$E(G)=\mathsf d(G)+|G|=p^3+3p-3.$$
\end{theorem}

Together with Gao's theorem for finite abelian groups and the
cyclic-index result \cite[Corollary~4.5]{QuWangLi},
Theorem~\ref{theorem:main} completes the verification of the
Gao-Zhuang conjecture for every group of order $p^3$. We record this
in Corollary~\ref{corollary:orderp3}.

Our proof makes use of classical zero-sum results on subsequences of prescribed lengths, together with some elegant ideas employed in Volkmann's argument. In particular, we use the ordering-value method for mixed zero-sum subsequences in $H_{p^3}/Z(H_{p^3})\cong C_p^2$. Combining these tools with a maximal decomposition into $p$-term zero-sum subsequences in this quotient, we are able to find a product-one subsequence of the desired length.

The remainder of the paper is organized as follows. Section~2 introduces the notation and terminology and recalls the basic structure of the Heisenberg group. Section~3 collects the necessary results and proves the auxiliary lemmas needed for the main argument. Section~4 is devoted to the proof of Theorem~\ref{theorem:main}.

\section{Notation and terminology}

Throughout this paper, all groups are finite. Unless additive notation
is explicitly used, the identity element of a group $G$ is denoted by
$1$, and the multiplication in $G$ is denoted by $\ast$. Thus,
$g\ast h$ denotes the product of $g,h\in G$. For a subset
$A\subseteq G$, the subgroup generated by $A$ is denoted by
$\langle A\rangle$. We write $C_n$ for a cyclic group of order $n$. For integers
$a\le b$, we put $[a, b]=\{a, a+1,\ldots, b\}$.

We use the standard language of sequences over groups;  see \cite{GaoGeroldingersurvey}
 and \cite[Chapter 5]{GHK}. For the
nonabelian product-one setting and the associated Davenport constants,
see \cite{GeroldingerGrynkiewicz2013} for example. A sequence over
$G$ is an element of the free abelian monoid $\mathcal F(G)$ and may be displayed either as
$$\mathop{\bullet}\limits_{g\in G} g^{[\mathsf v_g(S)]}
\qquad\text{or as}\qquad
S=g_1\boldsymbol{\cdot}\ldots\boldsymbol{\cdot}g_\ell,$$
where $\mathsf v_g(S)\in\mathbb N_0$ is the multiplicity of $g$ in $S$
and only finitely many multiplicities are nonzero. Here $\boldsymbol{\cdot}$ denotes multiplication in
$\mathcal F(G)$, that is, the concatenation of sequences, whereas
$\ast$ denotes multiplication in the group $G$. For $g\in G$ and $k\in\mathbb N_0$, the notation $g^{[k]}$ denotes
the sequence consisting of $k$ copies of $g$, whereas $g^k$ denotes
the $k$th power of $g$ in the group $G$.
The {\sl length}, {\sl support} and  {\sl height} of a sequence $S$ are
$|S|=\sum_{g\in G}\mathsf v_g(S)$, $\supp(S)=\{g\in G:\mathsf v_g(S)>0\}$ and $\mathsf h(S)=\max_{g}\mathsf v_g(S)$, respectively. The empty sequence has length
zero. For $S,T\in\mathcal F(G)$, we call $T$ a subsequence of $S$ and  write $T\mid S$ if
$\mathsf v_g(T)\leq\mathsf v_g(S)$ for every $g\in G$.  In this case,
$S\boldsymbol{\cdot}T^{[-1]}$ denotes the sequence remaining after the
terms of $T$ have been deleted from $S$. Subsequences $T_1,\ldots,T_k$ of
$S$ are called mutually disjoint if
$T_1\boldsymbol{\cdot}\ldots\boldsymbol{\cdot}T_k\mid S.$
For $X\subseteq G$, we denote by $S_X$ the subsequence of $S$
consisting of all terms from $X$.

A homomorphism $\varphi:G\to K$ acts termwise on sequences.  Namely, if
$S=g_1\boldsymbol{\cdot}\ldots\boldsymbol{\cdot}g_\ell$, then
$\varphi(S)=\varphi(g_1)\boldsymbol{\cdot}\ldots
\boldsymbol{\cdot}\varphi(g_\ell)\in\mathcal F(K).$
We use the same symbol $\varphi$ for the resulting monoid homomorphism
$\mathcal F(G)\to\mathcal F(K)$.

Let $S=g_1\boldsymbol{\cdot}\ldots\boldsymbol{\cdot}g_\ell$.  Because
$G$ need not be abelian, the product of all terms of $S$ may depend on
their order.  We therefore put
$$\pi(S)=
\bigl\{
g_{\tau(1)}\ast \cdots\ast g_{\tau(\ell)}:
\tau\ \text{is a permutation of }[1,\ell]
\bigr\}.$$
For $k\in[1,\ell]$, define
$$\Pi_k(S)=
\bigcup_{\substack{T\mid S\\ |T|=k}}\pi(T),
\qquad
\Pi(S)=\bigcup_{k=1}^{\ell}\Pi_k(S).$$
We call $S$ a \emph{product-one sequence} if
$1\in\pi(S)$, and \emph{product-one free} if $1\notin\Pi(S)$.

\begin{samepage}
\begin{definition}
Let $G$ be a finite group. The exponent of $G$ is
$\exp(G)=\operatorname{lcm}\{\ord(g):g\in G\}.$ The small Davenport constant $\mathsf d(G)$ and the Gao constant $E(G)$ of $G$ are defined as follows.
\begin{itemize}
\item[$\bullet$] $\mathsf d(G)$ is the largest integer $\ell$ such that
there exists a product-one free sequence over $G$ of length $\ell$;

\item[$\bullet$] $E(G)$ is the least integer $\ell$ such that every
sequence over $G$ of length at least $\ell$ contains a product-one
subsequence of length exactly $|G|$.

\end{itemize}
\end{definition}
\end{samepage}

We now define the Heisenberg group. Let $p$ be an odd prime, and let $\mathbb F_p$ denote the finite field with $p$ elements.
We realize the Heisenberg group as
$$H_{p^3}=\operatorname{UT}_3(\mathbb F_p)
=
\left\{
M(a,b,c)=
\begin{pmatrix}
1&a&c\\
0&1&b\\
0&0&1
\end{pmatrix}
:a,b,c\in\mathbb F_p
\right\}$$
with multiplication
$M(a,b,c)\ast M(a',b',c') = M(a+a',b+b',c+c'+ab')$.
Its center is
\begin{equation}\label{equation:centerisCp}
Z=Z(H_{p^3})= \{M(0,0,c):c\in\mathbb F_p\}\cong C_p.
\end{equation}
Let
\begin{equation}\label{equation specializehomo}
\varphi:H_{p^3}\longrightarrow H_{p^3}/Z\cong C_p^2
\end{equation}
be the canonical epimorphism. Then
$\varphi(M(a,b,c))=(a,b)$.

By the Correspondence Theorem applied to $\varphi$, every subgroup
$K$ of $H_{p^3}$ containing $Z$ and having order $p^2$ is the
preimage under $\varphi$ of a subgroup of order $p$ in
$H_{p^3}/Z\cong C_p^2$. Since $H_{p^3}$ has exponent $p$, it follows
that $K\cong C_p^2$.

Throughout the remainder of the paper,
$\varphi$ will always denote the canonical epimorphism given in
\eqref{equation specializehomo}.

Recall that a sequence in $\mathcal F(G)$ is unordered as defined above.  When the order
of its terms is relevant, we use parentheses to indicate a fixed
ordering. More precisely, if
$B=u_1\boldsymbol{\cdot}\ldots\boldsymbol{\cdot}u_n\in\mathcal F(G)$,
then an expression
$I=(u_{\tau(1)},\ldots,u_{\tau(n)})$, where
$\tau\in\mathfrak S_n$, will be called an \emph{ordering} of $B$.
Thus $I$ records an ordered list of the terms of $B$, while $B$
denotes the corresponding unordered sequence.

Let $V=\mathbb F_p^2$, and for $u\in V$ write
$u=(a(u),b(u))$. Let $B$ be a sequence over $V$. For an ordering
$I=(u_1,\ldots,u_n)$ of $B$, define its \emph{ordering value} by
$$q(I)=\sum_{1\leq i<j\leq n}a(u_i)b(u_j).$$
The \emph{ordering-value set} of $B$ is
$$\Omega(B)=\{q(I):I\text{ is an ordering of }B\}.$$
We call a sequence over $V$ \emph{mixed} if its support
is not contained in a subgroup of order $p$ of $V$.

In what follows, we use additive notation for the quotient
$H_{p^3}/Z(H_{p^3})\cong C_p^2$ and, when convenient, for abelian
subgroups of $H_{p^3}$. For a sequence $S$ over an abelian group
written additively, $\sigma(S)$ denotes the sum of its terms, and
for each positive integer $k$, we put
$$\Sigma_k(S)=\{\sigma(T):T\mid S,\ |T|=k\}.$$
In particular, $\Sigma_k(S)=\emptyset$ if $k>|S|$.
This is only a change of notation for the group operation;
product-one subsequences correspond to zero-sum subsequences.

\section{Some necessary lemmas}

We begin with the following lemmas.

\begin{lemma}{\cite[Theorem 3.2]{Volkmann}}\label{lemma:growth}
Let $p$ be an odd prime, and let $B$ be a mixed zero-sum sequence of
$n\ge3$ nonzero terms from $C_p^2$. Then
$|\Omega(B)|\ge \min(p,n-1)$.
\end{lemma}

\begin{lemma}\cite[Lemma~6.1]{Volkmann}\label{lemma:criterion}
Let
$T=g_1\boldsymbol{\cdot}\ldots\boldsymbol{\cdot}g_m$
be a sequence over $H_{p^3}$, where
$g_i=M(a_i,b_i,c_i)$ for every $i\in[1,m]$. If
$I=(g_{\tau(1)},\ldots,g_{\tau(m)})$
is an ordering of $T$, then
$$\prod_{i=1}^m g_{\tau(i)}
=
M\left(
\sum_{i=1}^m a_i,\,
\sum_{i=1}^m b_i,\,
\sum_{i=1}^m c_i+
\sum_{1\leq i<j\leq m}
a_{\tau(i)}b_{\tau(j)}
\right).$$
Consequently, suppose that $\varphi(T)$ is a zero-sum sequence over
$H_{p^3}/Z$, and put
$c(T)=\sum_{i=1}^m c_i\in\mathbb F_p$. Then
$\pi(T)\subseteq Z$ and
$\pi(T)=
\left\{
M(0,0,c(T)+\lambda):
\lambda\in\Omega(\varphi(T))
\right\}.$ In particular, if $\Omega(\varphi(T))=\mathbb{F}_p$ then $T$ is product-one.
\end{lemma}

The following prescribed-length theorem is due to W. Gao.

\begin{lemma}\cite[Theorem 3.2]{GaoRestrictedII}\label{lemma:prescribed}
Let $A$ be a finite abelian group of exponent $n$, and let $k$ be a
positive integer with $kn\ge |A|$. Every sequence over $A$ of length
at least
$kn+\mathsf d(A)$
contains a zero-sum subsequence of length $kn$.
\end{lemma}

In particular, since
$\mathsf d(C_p)=p-1$ and $\mathsf d(C_p^2)=2p-2$, every sequence
over $C_p$ of length at least $p^3+p-1$ contains a zero-sum
subsequence of length $p^3$, while every sequence over $C_p^2$ of
length at least $p^3+2p-2$ contains a zero-sum subsequence of length
$p^3$.

\begin{lemma}[{\cite[Lemma 4]{ZG2005}}]\label{lemma:lowerbound}
For every finite group $G$,
$E(G)\ge \mathsf d(G)+|G|$.
\end{lemma}

\begin{lemma}{\cite[Theorem 6.2]{Volkmann}}\label{lemma:dG}
For every odd prime $p$,
$\mathsf d(H_{p^3})=3p-3$.
\end{lemma}

\begin{lemma}[{\cite[Lemma 2.9]{QuGaoLi}}]\label{lemma:center}
Let $G$ be a finite group, and let $S$ be a sequence over $G$ of
length $|G|+\mathsf d(G)$. If
$\mathsf v_z(S)\ge\mathsf d(G)$
for some $z\in Z(G)$, then $S$ has a product-one subsequence of
length $|G|$.
\end{lemma}

\begin{lemma} (see \cite{Reiher}, or \cite[Theorem 4.2.10]{GRuzsa}) \label{lemma:ranktwoEGZ}
Every sequence of length $4p-3$ over $C_p^2$ contains a zero-sum
subsequence of length $p$.
\end{lemma}

\begin{lemma}\cite[Theorem~6.7(2)]{GaoGeroldingersurvey}\label{lemma:ranktwoshort}
Every sequence of length $3p-2$ over $C_p^2$ contains a zero-sum
subsequence of length $p$ or $2p$.
\end{lemma}

\begin{lemma}[Cauchy-Davenport Theorem]\cite[Theorem 2.3]{Nathanson}\label{lemma:Cauchy-Davenport}
Let $h\ge 2$. Let $p$ be a prime number, and let $A_1,\ldots,A_h$ be nonempty subsets of
$C_p$. Then
$|A_1+\cdots+A_h|\ge
\min\left(p, \sum\limits_{i=1}^h |A_i|-h+1 \right).$
\end{lemma}

Then we shall derive the following lemma by applying the Cauchy-Davenport Theorem.

\begin{lemma}\label{lemma:height}
Let $A$ be a sequence over $C_p$ such that
$|A|=p+k$, $\mathsf h(A)=p$ and $0\le k\le p-1$.
Then
$|\Sigma_p(A)|\ge k+1$.
\end{lemma}

\begin{proof} Let $g\in C_p$ be such that $\mathsf v_g(A)=p$. Since
$|A|=p+k$ with $0\leq k\leq p-1$, we may partition all the terms of
$A$ into $p$ subsequences $A_1,\ldots,A_p$ such that each $A_i$
contains exactly one copy of $g$ and $|A_i|\in \{1,2\}$.
Moreover, if $|A_i|=2$, then the other term of $A_i$ is different
from $g$. Thus each $A_i$ may be regarded as a nonempty subset of
$C_p$. Clearly,
$A_1+\cdots+A_p\subseteq\Sigma_p(A).$
By Lemma \ref{lemma:Cauchy-Davenport}, we have
$|\Sigma_p(A)|
\ge |A_1+\cdots+A_p|
\ge
\min\left(p,\sum_{i=1}^p|A_i|-p+1\right)
=\min(p,k+1)
=k+1.$
\end{proof}

We recall the notion of a multilinear polynomial. A polynomial
$f(X_1,\ldots,X_N)\in\mathbb F_p[X_1,\ldots,X_N]$ is called \emph{multilinear}
if its degree with respect to each indeterminate $X_i$ is at most $1$.
Equivalently, $f$ can be written in the form
\begin{equation}\label{equation:formf}
f(X_1,\ldots,X_N)=
\sum_{I\subseteq[1,N]}
c_I\prod_{i\in I}X_i \mbox{ with } c_I\in\mathbb F_p.
\end{equation}

\begin{lemma}\label{lemma:multilinearreduction} Let $N\ge 1$, and let $p$ be a prime. Let $f(X_1,\ldots,X_N)\in \mathbb{F}_p[X_1,\ldots,X_N]$ be a multilinear polynomial. If $f(X_1,\ldots,X_N)$ vanishes at every point $x\in\{0,1\}^N$, then $f$ is the zero polynomial. In particular, if $g\in \mathbb{F}_p[X_1,\ldots,X_N]$ is a multilinear polynomial such that $f(x)=g(x)$ for every $x\in\{0,1\}^N$, then $f=g$.
\end{lemma}

\begin{proof} Suppose that $f$ is not the zero polynomial. Since $f(\mathbf 0)=0$,
the constant term of $f$ is zero. Hence there exists a nonempty subset
$I\subseteq[1,N]$ such that
$c_I\prod_{i\in I}X_i$ occurs in \eqref{equation:formf} with
$c_I\ne0$. Choose such an $I$ minimal with respect to inclusion.

Let $x=(x_1,\ldots,x_N)\in\{0,1\}^N$ be defined by
$x_i=1$ if $i\in I$ and $x_i=0$ otherwise. Then
$$f(x)=\sum_{J\subseteq I}c_J.$$
By the minimality of $I$, we have $c_J=0$ for every proper subset
$J\subsetneq I$, and therefore $f(x)=c_I\ne0$. This contradicts the
hypothesis that $f$ vanishes at every point of $\{0,1\}^N$.
Hence $f$ is the zero polynomial.

For the last assertion, apply the first part to $f-g$.
\end{proof}

The following lemma is the main additive ingredient in the proof.
The hypothesis $N<3p$ is exactly the range needed below.

\begin{lemma}\label{lemma:relative}
Let
$A=a_1\boldsymbol{\cdot}\ldots\boldsymbol{\cdot}a_N$
be a sequence over $V=C_p^2$ with $N<3p$, and let $H\le V$ be a
subgroup of order $p$. Suppose that $A$ contains no zero-sum
subsequence of length $p$ or $2p$. Then
$\left|\left(\Sigma_p(A)\cup\Sigma_{2p}(A)\right)\cap H\right|
\ge N-2p+2.$
\end{lemma}

\begin{proof} We may assume that $N\ge 2p-1$, since otherwise
$N-2p+2\leq0$, and hence the conclusion follows trivially.

Write $V=H\bigoplus K$ where $K\simeq C_p$. Let $\theta_1$ and $\theta_2$ be the canonical projections of $V$ onto the components $H$ and $K$, respectively.  Then we consider the polynomial ring $\mathbb F_p[X_1,\ldots,X_{N}]$ over the field $\mathbb F_p$ with $N$ indeterminates. Since $C_p$ is isomorphic to the additive group of the field $\mathbb{F}_p$, we can view $\theta_i$ as the map of $V$ to $\mathbb{F}_p$ for $i\in [1,2]$. Then we take the following three polynomials $L(X_1,\ldots,X_N), Q(X_1,\ldots,X_N), J(X_1,\ldots,X_N)\in \mathbb F_p[X_1,\ldots,X_{N}]$, where
$$\begin{aligned}
&L(X_1,\ldots,X_N)=\sum_i\theta_1(a_i)X_i,  \\ 
&Q(X_1,\ldots,X_N)=\sum_i \theta_2(a_i)X_i, \\  
&J(X_1,\ldots,X_N)=\sum_{i=1}^N X_i. \end{aligned}
$$

For $x=(x_1,\ldots,x_N)\in\{0,1\}^N$, let $A_x$ denote the
subsequence of $A$ consisting of all terms $a_i$ whenever $x_i=1$, and we write
$L(x)=L(x_1,\ldots,x_N)$, and similarly for $J(x)$ and $Q(x)$.
Then we define
\begin{equation}\label{equation:Rdefinition}
R=\{L(x):x\in\{0,1\}^N\setminus\{\mathbf 0\},
      \ J(x)=Q(x)=0\}.
\end{equation}

Now we show that
\begin{equation}\label{equation:0notinR}
0\notin R,
\end{equation}
that is $R\subseteq\mathbb F_p\setminus \{0\}$.
Assume to the contrary that $0\in R$. Then there exists
$x=(x_1,\ldots,x_N)\in\{0,1\}^N\setminus\{\mathbf 0\}$ such that
$L(x)=J(x)=Q(x)=0$. Since $J(x)=0$ in $\mathbb F_p$ and $N<3p$,
the number of indices $i\in [1,N]$ with $x_i=1$ is either $p$ or $2p$.
Then $|A_x|\in\{p,2p\}$. Moreover,
$L(x)=Q(x)=0$ implies that
$\sigma(A_x)=0$ in $V=H\oplus K$. Thus $A$ contains a zero-sum
subsequence of length $p$ or $2p$, contradicting the hypothesis.
This proves \eqref{equation:0notinR}.

Then we take the univariate polynomial
\begin{equation}\label{equation:P(Y)definition}
P(Y)=\prod_{r\in R}(Y-r)\in\mathbb F_p[Y],
\end{equation}
where the empty product is understood to be the constant polynomial $1$.
We further take the polynomial $F(X_1,\ldots,X_N) \in \mathbb{F}_p[X_1,\ldots,X_N]$, where \begin{equation}\label{equation:Fdefinition}
F(X_1,\ldots,X_N)=
\left(1-J(X_1,\ldots,X_N)^{p-1}\right)
\left(1-Q(X_1,\ldots,X_N)^{p-1}\right)
P(L(X_1,\ldots,X_N)).
\end{equation}

We claim that
\begin{equation}\label{equation:F(x)=0}
F(x)=0\mbox{ for every } x=(x_1,\ldots,x_N)\in\{0,1\}^N\setminus\{\mathbf 0\}.
\end{equation}
Let $x\in\{0,1\}^N\setminus\{\mathbf 0\}$. If $J(x)\ne0$ or
$Q(x)\ne0$ in $\mathbb F_p$, then, by Fermat's
little theorem, the first or the second factor in
\eqref{equation:Fdefinition} vanishes, respectively. Hence, we may assume that $J(x)=Q(x)=0$. By
\eqref{equation:Rdefinition}, we then have $L(x)\in R$, and therefore
$P(L(x))=0$ by \eqref{equation:P(Y)definition}. Thus the third factor
in \eqref{equation:Fdefinition} vanishes. Consequently, $F(x)=0$. This proves \eqref{equation:F(x)=0}.

On the other hand, since $J(\mathbf 0)=Q(\mathbf 0)=L(\mathbf 0)=0$, it follows from \eqref{equation:0notinR}, \eqref{equation:P(Y)definition} and \eqref{equation:Fdefinition} that $$F(\mathbf 0)=P(L(\mathbf 0))=P(0)\neq 0.$$ Combined with \eqref{equation:F(x)=0}, we derive that $F(X_1,\ldots,X_N)$ and
$P(0)\prod_{i=1}^N(1-X_i)$ have the same value at every point of
$\{0,1\}^N$.

Since $x_i^2=x_i$ for every $x_i\in\{0,1\}$, by taking the remainder modulo $X_i^2-X_i$ for each $i\in[1,N]$,
we obtain a multilinear polynomial
$\widetilde F(X_1,\ldots,X_N)\in
\mathbb F_p[X_1,\ldots,X_N]$
such that
$$\widetilde F(x)=F(x)
\qquad\text{for every }x\in\{0,1\}^N.$$
The polynomial $P(0)\prod_{i=1}^N(1-X_i)$ is multilinear and has the
same values as $\widetilde F$ at every point of $\{0,1\}^N$. Hence,
by Lemma~\ref{lemma:multilinearreduction},
$$\widetilde F(X_1,\ldots,X_N)=P(0)\prod_{i=1}^N(1-X_i).$$
Since $P(0)\ne0$, we have $$\deg\widetilde F=N.$$ By \eqref{equation:P(Y)definition} and the definition of $L(X_1,\ldots,X_N)$, we have that $$\deg P(L(X_1,\ldots,X_N))\leq \deg P(Y)=|R|.$$
Moreover, taking the remainder
modulo $X_i^2-X_i$ cannot increase the total degree. Therefore, by
\eqref{equation:Fdefinition}, we conclude that
$N=\deg\widetilde F\leq\deg F\leq(p-1)+(p-1)+|R|=2p-2+|R|$.
Consequently,
\begin{equation}\label{equation:Rbound}
|R|\ge N-2p+2.
\end{equation}

Finally, we identify the set $R$.  If
$x=(x_1,\ldots,x_N)\in\{0,1\}^N\setminus\{\mathbf 0\}$ satisfies $J(x)=0$, then
$0<|A_x|<3p$ and $|A_x|\equiv0\pmod p$, so
$|A_x|\in\{p,2p\}$. Moreover, $Q(x)=0$ means that
$\sigma(A_x)\in H$, and in this case
$L(x)=\sigma(A_x)$. Hence $R=
\left(\Sigma_p(A)\cup\Sigma_{2p}(A)\right)\cap H.$
Together with \eqref{equation:Rbound}, this gives $\left|
\left(\Sigma_p(A)\cup\Sigma_{2p}(A)\right)\cap H
\right|
\ge N-2p+2,$
as desired.
\end{proof}

\begin{lemma}\label{lemma:upperlength}
Let $S$ be a sequence over $G=H_{p^3}$ of length
$|S|=p^3+3p-3$
and suppose that $S$ has no product-one subsequence of length $p^3$.
Then
$|S_K|\le p^3+2p-3$
for every subgroup $K\supseteq Z$ of order $p^2$.
\end{lemma}

\begin{proof}
We have $K\cong C_p^2$. If
$|S_K|\ge p^3+2p-2,$
then Lemma~\ref{lemma:prescribed}, applied to $K$, gives a product-one
subsequence of $S_K$ of length $p^3$, a contradiction.
\end{proof}

\section{Proof of Theorem \ref{theorem:main}}

\begin{proof}[Proof of Theorem \ref{theorem:main}]
By Lemma~\ref{lemma:lowerbound} and Lemma~\ref{lemma:dG}, we have $\mathsf d(G)=3p-3$ and
$E(G)\ge \mathsf d(G)+|G|=p^3+3p-3$.
It remains to prove $E(G)\leq p^3+3p-3$. Let $S$ be a sequence over
$G$ of length
$|S|=p^3+3p-3$
and suppose, to the contrary, that
$1\notin\Pi_{p^3}(S)$.

By Lemma~\ref{lemma:center},
$\mathsf v_z(S)\leq \mathsf d(G)-1=3p-4$ for every $z\in Z$. By \eqref{equation:centerisCp}, we see $|Z|=p$, and thus,
\begin{equation}\label{equation:centerbound}
|S_Z|\le p(3p-4)=3p^2-4p.
\end{equation}

Let
$\varphi:H_{p^3}\longrightarrow H_{p^3}/Z\cong C_p^2$
be the canonical epimorphism given as \eqref{equation specializehomo}. Then we show the following claim.

\medskip
\noindent\textbf{Claim 1.}
If $T\mid S$ has length $p^3$ and $\varphi(T)$ is a mixed zero-sum
sequence over $G/Z$, then
$1\in\pi(T)$.

 \smallskip

\noindent {\sl Proof of Claim 1.}
Put
$T_0=T\boldsymbol{\cdot}T_Z^{[-1]}$.
By \eqref{equation:centerbound}, we have
\begin{equation}\label{equation:T0lengthget2p+1}
|T_0|=|T|-|T_Z|=p^3-|T_Z|\ge p^3-|S_Z| \ge 4p.
\end{equation}
 Since every term of $T_Z$ lies in $\ker\varphi=Z$, we have
$\varphi(T)=\varphi(T_0)\boldsymbol{\cdot}0^{[|T_Z|]}.$
Hence $\sigma(\varphi(T_0))=\sigma(\varphi(T))=0$. Since every subgroup of $G/Z$ contains $0$, removing the zero terms
does not affect mixedness.
Therefore, $\varphi(T_0)$ is a mixed
zero-sum sequence over $G/Z$, all of whose terms are nonzero.

Hence, by \eqref{equation:T0lengthget2p+1} and Lemma~\ref{lemma:growth}, we derive that
$|\Omega(\varphi(T_0))|\ge  \min(p,|\varphi(T_0)|-1)=\min(p,|T_0|-1)=p$. Since $\Omega(\varphi(T_0))\subseteq\mathbb F_p$, it follows that
$\Omega(\varphi(T_0))=\mathbb F_p$.
Since
$\varphi(T)=\varphi(T_0)\boldsymbol{\cdot}0^{[|T_Z|]}$ and
the zero terms do not contribute to the set $\Omega(\varphi(T))$, we have
$\Omega(\varphi(T))=
\Omega(\varphi(T_0))=
\mathbb F_p.$
Thus, Claim~1 follows from  Lemma~\ref{lemma:criterion} immediately. \qed

\medskip
Let $\ell$ be maximal such that
$S=T_1\boldsymbol{\cdot}\ldots\boldsymbol{\cdot}T_\ell \boldsymbol{\cdot}S',$
where
$|T_i|=p$
and $\varphi(T_i)$ is zero-sum over $G/Z$ for every $i\in[1,\ell]$.
By Lemma~\ref{lemma:ranktwoEGZ} and the  maximality of $\ell$, we have
$|S'|\le4p-4$.
Therefore
$p^3+3p-3-p\ell\le4p-4,$
and hence
\begin{equation}\label{equation:ell}
\ell\ge p^2.
\end{equation}

\medskip
\noindent\textbf{Claim 2.} There exists a subgroup $\overline K$ of order $p$ in $G/Z$ such
that
$\supp\left(\varphi(T_1\boldsymbol{\cdot}\ldots
\boldsymbol{\cdot}T_\ell)
\right)\subseteq\overline K.$

 \smallskip

\noindent {\sl Proof of Claim 2.} Suppose otherwise.
Then there exist two linearly
independent elements $u,v\in \supp(\varphi(T_1\boldsymbol{\cdot}\ldots\boldsymbol{\cdot}T_\ell))$.
Combined with \eqref{equation:ell}, we may choose a subset
$I$ of $[1,\ell]$ such that $|I|=p^2$
and $u,v\in \supp(\varphi(U))$, where $U=\mathop{\boldsymbol{\cdot}}\limits_{k\in I}T_k.$
Since each $\varphi(T_k)$ is zero-sum,  $\varphi(U)$ is also zero-sum.
Moreover, since $u,v\in \supp(\varphi(U))$, the linear
independence of $u$ and $v$ implies that
$\varphi(U)$ is mixed. Since
$|U|=p^3$, Claim~1 gives $1\in\pi(U)$, contradicting
the assumption that $1\notin\Pi_{p^3}(S)$. This proves Claim 2. \qed

Fix a subgroup $\overline K$ as in Claim~2 and put
$K=\varphi^{-1}(\overline K)$.
Then $K\supseteq Z$, $|K|=p^2$, and
$T_1\boldsymbol{\cdot}\ldots\boldsymbol{\cdot}T_\ell\mid S_K$.
Hence, by \eqref{equation:ell} and Lemma~\ref{lemma:upperlength},
\begin{equation}\label{equation:SKfirst}
p^3\le |S_K|\le p^3+2p-3.
\end{equation}

We next give a consequence of Claim~1 that will be used several
times.

\medskip
\noindent\textbf{Claim 3.}
Let $V\mid S$ be such that $|V|=p^3$ and $\varphi(V)$ is a zero-sum
sequence over $G/Z$. Then
$V\mid S_K.$

\smallskip

\noindent {\sl Proof of Claim 3.}
Since $1\notin\Pi_{p^3}(S)$, it follows from Claim~1 that
$\varphi(V)$ is not mixed. Suppose, to the contrary, that
$V\nmid S_K$. Then $\varphi(V)$ contains a nonzero term outside
$\overline K$. Since $\varphi(V)$ is not mixed, there exists a
subgroup $\overline H$ of order $p$ in $G/Z$ such that
$$
\supp(\varphi(V))\subseteq\overline H
\qquad\text{and}\qquad
\overline H\ne\overline K.
$$
Put $H=\varphi^{-1}(\overline H)$. Then $V\mid S_H$,
$|H|=p^2$, and $H\cap K=Z$. Hence, by
\eqref{equation:centerbound} and \eqref{equation:SKfirst},
$$\begin{aligned}
p^3+3p-3=|S|
&\ge |S_{K\setminus Z}|+|S_{H\setminus Z}|+|S_Z|\\
&=(|S_K|-|S_Z|)+(|S_H|-|S_Z|)+|S_Z|\\
&=|S_K|+|S_H|-|S_Z|\\
&\ge |S_K|+|V|-|S_Z|\\
&\ge 2p^3-(3p^2-4p).
\end{aligned}$$
It follows that
$p^3-3p^2+p+3\leq0,$
which is impossible for $p\ge3$. Therefore $V\mid S_K$, proving
Claim~3. \qed

Since
$|S_{K\setminus Z}| = |S_K|-|S_Z| \ge p^3-(3p^2-4p) \ge 4p$, we choose an  arbitrary subsequence
$$W\mid S_{K\setminus Z}$$ of length $|W|=p-1$.
Let $t$ be maximal such that
\begin{equation}\label{equation:S=WUiS''}
S= W\boldsymbol{\cdot} U_1\boldsymbol{\cdot}\ldots\boldsymbol{\cdot}U_t \boldsymbol{\cdot}S'',
\end{equation}
where $|U_i|=p$ and $\varphi(U_i)$ is zero-sum over $G/Z$ for every
$i\in[1,t]$. By the maximality of $t$ and
Lemma~\ref{lemma:ranktwoEGZ}, we have
$|S''|\le4p-4.$ Hence
$p^3+2p-2-tp=|S''|\leq4p-4,$
and therefore
\begin{equation}\label{equation:tlower}
t\ge
\left\lceil\frac{p^3-2p+2}{p}\right\rceil
=p^2-1.
\end{equation}

We next prove that
\begin{equation}\label{equation:WinK}
W\boldsymbol{\cdot}U_1\boldsymbol{\cdot}\ldots
\boldsymbol{\cdot}U_t\mid S_K.
\end{equation}
Since $W\mid S_{K\setminus Z}$, it suffices to show that $U_1\boldsymbol{\cdot}\ldots
\boldsymbol{\cdot}U_t\mid S_K$. Now suppose that $t\ge p^2$. Take an arbitrary $j\in[1,t]$. Choose a subset
$I\subseteq[1,t]$ such that $|I|=p^2$ and $j\in I$. Put
$V_I=
\mathop{\boldsymbol{\cdot}}\limits_{i\in I}U_i.$
Then $V_I\mid S$, $|V_I|=p^3$, and $\varphi(V_I)$ is zero-sum.
By Claim~3, $V_I\mid S_K$, and hence $U_j\mid S_K$. Since $j$ was
arbitrary, we obtain
$U_1\boldsymbol{\cdot}\ldots\boldsymbol{\cdot}U_t\mid S_K$ in this case. Hence, by \eqref{equation:tlower}, to prove \eqref{equation:WinK}, it remains to consider the case  $$t=p^2-1.$$ By
\eqref{equation:S=WUiS''}, we have
$$
\begin{aligned}
|W\boldsymbol{\cdot}S''|
&=|S|-
\left|U_1\boldsymbol{\cdot}\ldots
\boldsymbol{\cdot}U_t\right|\\
&=(p^3+3p-3)-(p^2-1)p\\
&=4p-3.
\end{aligned}
$$
By Lemma~\ref{lemma:ranktwoEGZ}, there exists a subsequence
$W'\mid W\boldsymbol{\cdot}S''$ such that $|W'|=p$ and
$\varphi(W')$ is zero-sum. Put
$V=
W'\boldsymbol{\cdot}U_1\boldsymbol{\cdot}\ldots
\boldsymbol{\cdot}U_t.$
Then $V\mid S$, $|V|=p^3$, and $\varphi(V)$ is zero-sum. By
Claim~3, $V\mid S_K$.  This proves
\eqref{equation:WinK} in all cases.

We now show that
\begin{equation}\label{equation:t=p2}
t=p^2.
\end{equation}
By \eqref{equation:WinK} and
Lemma~\ref{lemma:upperlength}, we have
$(p-1)+tp
=
\left|
W\boldsymbol{\cdot}U_1\boldsymbol{\cdot}\ldots
\boldsymbol{\cdot}U_t
\right|
\leq |S_K|
\leq p^3+2p-3.$
Thus $t\leq p^2$. Together with \eqref{equation:tlower}, this gives
$$
t\in\{p^2-1,p^2\}.
$$
Suppose, to the contrary, that $t=p^2-1$. It follows from
\eqref{equation:S=WUiS''} that
$|S''|=|S|-|W|-\sum_{i=1}^t|U_i|=3p-2.$
By Lemma~\ref{lemma:ranktwoshort}, there exists a subsequence
$U\mid S''$ such that $|U|\in\{p,2p\}$ and $\varphi(U)$ is
zero-sum over $G/Z$. The maximality of $t$ excludes $|U|=p$.
Therefore $|U|=2p$. Put
$$V=
U_1\boldsymbol{\cdot}\ldots
\boldsymbol{\cdot}U_{t-1}
\boldsymbol{\cdot}U.$$
Then $V\mid S$, $|V|=p^3$, and $\varphi(V)$ is zero-sum over
$G/Z$. By Claim~3, $V\mid S_K$, and hence $U\mid S_K$. Since
$U\mid S''$, the subsequences
$W,U_1,\ldots,U_t,U$ are mutually disjoint. Consequently,
$|S_K|
\ge |W|+\sum_{i=1}^{t}|U_i|+|U|=(p-1)+(p^2-1)p+2p=p^3+2p-1,$
contrary to Lemma~\ref{lemma:upperlength}. This proves \eqref{equation:t=p2}.

Then, by \eqref{equation:WinK} and \eqref{equation:t=p2}, we derive that
$|S_K|
\ge\left|
W\boldsymbol{\cdot}U_1\boldsymbol{\cdot}\ldots\boldsymbol{\cdot}U_t\right|=(p-1)+p^3=p^3+p-1.$
Combining this with Lemma~\ref{lemma:upperlength}, we obtain
\begin{equation}\label{equation:t}
p^3+p-1\leq |S_K|\leq p^3+2p-3.
\end{equation}

Put
$$L=W\boldsymbol{\cdot}S''
\boldsymbol{\cdot}(S''_K)^{[-1]},$$ where $S''_K$ denotes the
subsequence of $S''$ consisting of all terms lying in $K$.
Since
$S=
W\boldsymbol{\cdot}U_1\boldsymbol{\cdot}\ldots
\boldsymbol{\cdot}U_t\boldsymbol{\cdot}S'',$
it follows from \eqref{equation:WinK} that
$S_K=
W\boldsymbol{\cdot}U_1\boldsymbol{\cdot}\ldots
\boldsymbol{\cdot}U_t\boldsymbol{\cdot}S''_K.$
Consequently,
\begin{equation}\label{equation:S=SKtwo}
S=S_K\boldsymbol{\cdot}
S''\boldsymbol{\cdot}(S''_K)^{[-1]}.
\end{equation}
On the other hand, by the definition of $L$, we have
$L\boldsymbol{\cdot}W^{[-1]}
=
S''\boldsymbol{\cdot}(S''_K)^{[-1]}.$
Therefore,
$S=
S_K\boldsymbol{\cdot}
\bigl(L\boldsymbol{\cdot}W^{[-1]}\bigr).$
Taking lengths, we obtain
$$
\begin{aligned}
|L|
&=|S|-|S_K|+|W|\\
&=(p^3+3p-3)-|S_K|+(p-1)\\
&=p^3+4p-4-|S_K|.
\end{aligned}
$$
Combining this equality with \eqref{equation:t}, we obtain
\begin{equation}\label{equation:Lrange}
2p-1\leq |L|\leq3p-3.
\end{equation}
Recall $W\mid S_{K\setminus Z}$. Since
$S''\boldsymbol{\cdot}(S''_K)^{[-1]}$ contains no terms from $K$,
the terms of $L$ lying in $K$ are
precisely the terms of $W$. Hence
\begin{equation}\label{equation:LK=W}
L_K=L_{K\setminus Z}=W.
\end{equation}

\medskip
\noindent\textbf{Claim 4.} The sequence $\varphi(L)$ contains no zero-sum subsequence of length
$p$ or $2p$.

 \smallskip

\noindent {\sl Proof of Claim 4.}
Suppose, to the contrary, that there exists a
subsequence $L'\mid L$ such that
$|L'|\in\{p,2p\}$ and $\varphi(L')$ is zero-sum. Put
$k=p^2-\frac{|L'|}{p}\in\{p^2-1,p^2-2\}$
and
$V'=U_1\boldsymbol{\cdot}\ldots
\boldsymbol{\cdot}U_k\boldsymbol{\cdot}L'.$
By \eqref{equation:S=WUiS''} and the definition of $L$, we have
$V'\mid S$. Moreover,
$|V'|=kp+|L'|=p^3,$
and $\varphi(V')$ is zero-sum. Hence Claim~3 gives
$V'\mid S_K$, and therefore $L'\mid S_K$, i.e., all terms of $L'$ lie in
$K$.  Since $L'\mid L$, it follows from \eqref{equation:LK=W} that $L'\mid W$. This is impossible because
$|L'|\ge p$, whereas $|W|=p-1$. This proves Claim 4. \qed

Apply Lemma~\ref{lemma:relative} to the sequence $\varphi(L)$ and
the subgroup
$\overline K=K/Z$ of $G/Z$. By \eqref{equation:Lrange}, we have
$|\varphi(L)|=|L|<3p$, and Claim~4 shows that $\varphi(L)$ contains
no zero-sum subsequence of length $p$ or $2p$. Hence all the
hypotheses of Lemma~\ref{lemma:relative} are satisfied. Put
$$
\mathcal D=
\left(
\Sigma_p(\varphi(L))
\cup
\Sigma_{2p}(\varphi(L))
\right)
\cap\overline K.
$$
Then
$$
|\mathcal D|\geq |L|-2p+2.
$$
By \eqref{equation:S=SKtwo} and the definition of the sequence $L$, we see that
\begin{equation}\label{equation:SL-1=SKW-1}
S\boldsymbol{\cdot}L^{[-1]}=
S_K\boldsymbol{\cdot}W^{[-1]}.
\end{equation}
Then
\begin{equation}\label{equation:A=varpi=varpi}
A:=\varphi\left(S\boldsymbol{\cdot}L^{[-1]}\right)=\varphi(S_K\boldsymbol{\cdot}W^{[-1]})
\end{equation}
is a sequence over $\overline K\cong C_p$ of length
$|A|=p^3+3p-3-|L|.$
By \eqref{equation:Lrange}, we have $|A|\geq p^3$. Since
$|\overline K|=p$, it follows that
$\mathsf h(A)
\geq
\left\lceil\frac{|A|}{p}\right\rceil
\geq p^2.$

Now we show that
\begin{equation}\label{equation:h(A)geqp3}
\mathsf h(A)\geq p^3.
\end{equation}
Suppose, to the contrary, that $\mathsf h(A)\leq p^3-1$. Put
$r=3p-|L|-2.$
By \eqref{equation:Lrange}, we have
$$
1\leq r\leq p-1
\qquad\text{and}\qquad
|A|=p^3+r-1.
$$
Choose $g\in\overline K$ such that
$\mathsf v_g(A)=\mathsf h(A)$. Then $A$ contains at least $p$
copies of $g$. Moreover,
$|A|-\mathsf v_g(A)
\geq
(p^3+r-1)-(p^3-1)
=r,$
so $A$ contains at least $r$ terms different from $g$. Hence we may
choose a subsequence
$$R\mid S\boldsymbol{\cdot}L^{[-1]}$$ such that
$$\varphi(R)
=
g^{[p]}\boldsymbol{\cdot}R_0,
\qquad
|R_0|=r,
\qquad
g\notin\supp(R_0).$$
Since $r\leq p-1$, it follows that
$|R|=p+r$ and
$\mathsf h(\varphi(R))=p.$
By Lemma~\ref{lemma:height},
$$|\Sigma_p(\varphi(R))|
\geq r+1
=3p-|L|-1.$$
Applying Lemma~\ref{lemma:Cauchy-Davenport} to subsets
$\Sigma_p(\varphi(R))$ and $\mathcal D$ of the group $\overline K\cong C_p$, we obtain
$$\begin{aligned}
\left|
\Sigma_p(\varphi(R))+\mathcal D
\right|
&\geq
\min\left(
p,\,
|\Sigma_p(\varphi(R))|+|\mathcal D|-1
\right)\\
&\geq
\min\left(
p,\,
(3p-|L|-1)+(|L|-2p+2)-1
\right)\\
&=p.
\end{aligned}$$
Therefore
$\Sigma_p(\varphi(R))+\mathcal D=\overline K.$
In particular, $0$ belongs to this sumset. Hence there exist a
$p$-term subsequence $R'\mid R$ and a subsequence $L'\mid L$ of
length $p$ or $2p$ such that
$$\sigma(\varphi(R'))+\sigma(\varphi(L'))=0.$$
Note that $\sigma(\varphi(R'))\ne0$, since otherwise, $\sigma(\varphi(L'))=0$ contradicts Claim 4. By \eqref{equation:SL-1=SKW-1}, we see
$R'\mid
S\boldsymbol{\cdot}L^{[-1]}
=
S_K\boldsymbol{\cdot}W^{[-1]},$
so all terms of $R'$ lie in $K$. Thus
$\sigma(\varphi(R'))\ne0$ implies that $R'$ contains a term from
$K\setminus Z$.
On the other hand,
$L'$ contains a term outside $K$, since
$|L'|\geq p$, whereas
$$
L_K=W
\qquad\text{and}\qquad
|W|=p-1.
$$
Therefore the support of
$\varphi(R'\boldsymbol{\cdot}L')$ contains a nonzero element of
$\overline K$ and an element outside $\overline K$. Put
$B=R'\boldsymbol{\cdot}L'.$ We see that $B$ is a subsequence of $S$ with length
$|B|\in\{2p,3p\}$
such that $\varphi(B)$ is a mixed zero-sum sequence.
Put
$b=\frac{|B|}{p}\in\{2,3\}.$
Choose a maximal family of
mutually disjoint $p$-term subsequences
$V_1,\ldots,V_s$ of $S\boldsymbol{\cdot}B^{[-1]}$ such that $\varphi(V_i)$ is zero-sum for every
$i\in[1,s]$. Let
$S_0=
S\boldsymbol{\cdot}B^{[-1]}
\boldsymbol{\cdot}
\left(
V_1\boldsymbol{\cdot}\ldots\boldsymbol{\cdot}V_s
\right)^{[-1]}.$
By the maximality of $s$ and Lemma~\ref{lemma:ranktwoEGZ}, we have
$|S_0|\leq4p-4$. Therefore,
$$
\begin{aligned}
sp
&=|S|-|B|-|S_0|\\
&\geq
(p^3+3p-3)-bp-(4p-4)\\
&=p^3-(b+1)p+1,
\end{aligned}
$$
which implies
$s\geq p^2-b.$
Consequently,
$B\boldsymbol{\cdot}
V_1\boldsymbol{\cdot}\ldots
\boldsymbol{\cdot}V_{p^2-b}$
is a subsequence of $S$ of length
$bp+(p^2-b)p=p^3.$
Its image under $\varphi$ is zero-sum and remains mixed, since it
contains the mixed sequence $\varphi(B)$. Claim~1 therefore gives a
product-one subsequence of $S$ of length $p^3$, a contradiction. This proves \eqref{equation:h(A)geqp3}.

Recall the arbitrariness of choosing $W\mid S_{K\setminus Z}$ subject to
$|W|=p-1$.  Combined with \eqref{equation:A=varpi=varpi} and \eqref{equation:h(A)geqp3}, we derive the following claim immediately.

\medskip

\noindent\textbf{Claim 5.} $\mathsf h\left(\varphi\left(S_K\boldsymbol{\cdot}W^{[-1]}\right)\right)\geq p^3$ holds for every $W\mid S_{K\setminus Z}$ such that $|W|=p-1$.

\medskip

Put
$C=\varphi(S_K).$
Then we have
$\mathsf h(C)\geq \mathsf h\left(
\varphi\left(
S_K\boldsymbol{\cdot}W^{[-1]}\right)\right)\geq  p^3$.
By \eqref{equation:SKfirst}, we have
$$|C|=|S_K|
\leq p^3+2p-3
<2p^3.$$
It follows that there exists a unique element
$g\in\overline K$ such that
$\mathsf v_g(C)\geq p^3.$ Moreover, by \eqref{equation:centerbound} we have
$\mathsf v_0(C)=|S_Z|
\leq3p^2-4p
<p^3$, which implies that $g\neq 0$.

We show that
\begin{equation}\label{equation:bigcoset}
\mathsf v_g(C)\geq p^3+p-1.
\end{equation}
Suppose, to the contrary, that
$p^3\leq \mathsf v_g(C)\leq p^3+p-2.$ We choose $\widetilde{W}\mid S_K$ such that $\varphi(\widetilde{W})=g^{[p-1]}$.
Since $g\neq 0$, we have $\widetilde{W}\mid S_{K\setminus Z}$.  Then $\mathsf v_g\left(
\varphi\left(
S_K\boldsymbol{\cdot}\widetilde{W}^{[-1]}
\right)
\right)=\mathsf v_g(C)-(p-1)\leq p^3+p-2-(p-1)=
p^3-1$. Moreover, every other element of $\overline K$ already has
multiplicity at most $p^3-1$.
Therefore,
$\mathsf h\left(
\varphi\left(
S_K\boldsymbol{\cdot}\widetilde{W}^{[-1]}
\right)
\right)
\leq p^3-1.$
By applying Claim 5 to $\widetilde{W}$, we derive a contradiction. This proves \eqref{equation:bigcoset}.

By \eqref{equation:bigcoset}, we can choose a subsequence
$y_1\boldsymbol{\cdot}\ldots
\boldsymbol{\cdot}y_{p^3+p-1}
\mid S_K$
such that $\varphi(y_i)=g$ for each
$i\in[1,p^3+p-1]$. Fix an element $x\in K$ such that
$\varphi(x)=g$. Since
$\ker(\varphi)=Z$, every $y_i$ can be written uniquely as
$y_i=x\ast z_i$ with $z_i\in Z.$
Thus
$z_1\boldsymbol{\cdot}\ldots
\boldsymbol{\cdot}z_{p^3+p-1}$
is a sequence over $Z\cong C_p$ of length $p^3+p-1$.
By Lemma~\ref{lemma:prescribed}, it contains a zero-sum subsequence
$z_{i_1}\boldsymbol{\cdot}\ldots
\boldsymbol{\cdot}z_{i_{p^3}}$
of length $p^3$. Under the identification of the additive group
$C_p$ with the multiplicative group $Z$, this means that
$z_{i_1}\ast\cdots \ast z_{i_{p^3}}=1.$
Since $K$ is abelian and has exponent $p$, we have
$y_{i_1}\ast \cdots \ast y_{i_{p^3}}
=(x\ast z_{i_1})\ast \cdots\ast (x\ast z_{i_{p^3}})
=x^{p^3}\ast z_{i_1}\ast \cdots \ast z_{i_{p^3}}
=1,$
contrary to $1\notin\Pi_{p^3}(S)$. This completes the proof of the theorem.
\end{proof}

We conclude with the following corollary of Theorem~\ref{theorem:main}.

\begin{corollary}\label{corollary:orderp3}
Let $p$ be a prime and let $G$ be a group of order $p^3$. Then
$$
E(G)=\mathsf d(G)+|G|.
$$
\end{corollary}

\begin{proof}
If $G$ is abelian, the conclusion follows from Gao's theorem
\cite{Gao1996}. Suppose that $G$ is nonabelian. Since
$\exp(G)\mid p^3$ and $\exp(G)=p^3$ would imply that $G$ is cyclic,
we have
$$
\exp(G)\in\{p,p^2\}.
$$

If $\exp(G)=p^2$, then $G$ contains an element $x$ of order $p^2$.
Thus $\langle x\rangle$ is a cyclic subgroup of index $p$ in $G$.
Since $p$ is the smallest prime divisor of $|G|$, the conclusion
follows from \cite[Corollary~4.5]{QuWangLi}.

It remains to consider the case $\exp(G)=p$. In this case $p$ is
odd, since every group of exponent $2$ is abelian. By
\cite[Chapter~5, Theorem~5.1]{Gorenstein}, up to isomorphism, the
unique nonabelian group of order $p^3$ and exponent $p$ is
$H_{p^3}=\operatorname{UT}_3(\mathbb F_p)$. Hence the conclusion follows from
Theorem~\ref{theorem:main}.
\end{proof}

\bigskip

\noindent {\bf Acknowledgments}

\noindent
This work was supported by the National Natural Science Foundation of
China (NSFC) under Grant No. 12371335, and supported by the Henan
Provincial Selective Research Funding Program for Returned Scholars
Studying Abroad (No. HNLX202611).

\section*{Author information}

\noindent\textbf{Yongke Qu}\\
Department of Mathematics, Luoyang Normal University,
Luoyang 471934, P.R. China\\
E-mail: \href{mailto:yongke1239@163.com}{yongke1239@163.com}

\medskip

\noindent\textbf{Guoqing Wang}\\
School of Mathematical Sciences, Tiangong University,
Tianjin 300387, P.R. China\\
E-mail: \href{mailto:gqwang1979@aliyun.com}{gqwang1979@aliyun.com}

\end{document}